%% file: cd.tex
\documentclass[preprint,12pt]{elsarticle}
\usepackage{graphics}
\usepackage{amsmath}
\usepackage{subfig}			 
\usepackage{stackrel}
\usepackage{hyperref}
\usepackage{natbib}
\usepackage{tikz}
\usepackage{mathrsfs}
\usepackage{enumerate}
\usepackage{amssymb}
\usepackage{amsthm}

\newdefinition{definition}{Definition}[section]
\newtheorem{proposition}[definition]{Proposition}
\newtheorem{lemma}[definition]{Lemma}
\newtheorem{theorem}[definition]{Theorem}
\newtheorem{corollary}[definition]{Corollary}

\def \calh {{\mathcal H}}
\def \hatphi {{\hat{\varphi}}}
\graphicspath{{./figs/}}

\journal{Discrete Mathematics}

\begin{document}

\newcommand{\comment}[1]{}

\begin{frontmatter}



\title{The chromatic discrepancy of graphs}


\author{N.R. Aravind}
\author{Subrahmanyam Kalyanasundaram}
\author{R.B. Sandeep\fnref{support}}
\fntext[support]{Supported by TCS Research Scholarship}
\author{Naveen Sivadasan}

\address{Department of Computer Science \& Engineering\\
Indian Institute of Technology Hyderabad, INDIA\\
\{\texttt{aravind, subruk, cs12p0001, nsivadasan}\}\makeatletter@\makeatother \texttt{iith.ac.in}}

\begin{abstract}

For a proper vertex coloring $c$ of a graph $G$, let $\varphi_c(G)$ denote the maximum, over all induced subgraphs $H$ of $G$, the difference between the chromatic number $\chi(H)$ and the number of colors used by $c$ to color $H$. We define the \textbf{chromatic discrepancy} of a graph $G$, denoted by $\varphi(G)$, to be the minimum $\varphi_c(G)$, over all proper colorings $c$ of $G$. If $H$ is restricted to only connected induced subgraphs, we denote the corresponding parameter by $\hatphi(G)$. These parameters are aimed at studying graph colorings that  use as few colors as possible in a graph and all its induced subgraphs. We study the parameters $\varphi(G)$ and $\hatphi(G)$ and obtain bounds on them. We obtain general bounds, as well as bounds for certain special classes of graphs including random graphs. We provide structural characterizations of graphs with $\varphi(G) = 0$ and graphs with $\hatphi(G) = 0$.
We also show that computing these parameters is NP-hard.

\end{abstract}


\end{frontmatter}



\section{Introduction}
\label{section:intro}
Consider a proper vertex coloring of a graph $G$, i.e., no two adjacent vertices get the same color.
Let $\chi(G)$ denote the \emph{chromatic number} of $G$.
It is natural to insist that every induced subgraph of $G$ also is colored with as few colors as possible.
In particular, for a coloring $c$ of $G$, let $\varphi_{c}(G)$ denote the maximum of, over all induced subgraphs $H$ of $G$, the difference between $\chi(H)$ and the number of colors used by $c$ to color $H$.
We define $\varphi(G)$, called the \emph{chromatic discrepancy} of $G$, as the minimum $\varphi_{c}(G)$ over all proper colorings $c$ of $G$.
We obtain several bounds on $\varphi(G)$ and study its relation to other graph parameters.

We consider only finite and simple undirected graphs. We follow basic definitions and
notations in \cite{douglas1996west}. For a graph $G$, we use $V(G)$, $E(G)$ and $|G|$ to denote the set of vertices of $G$, the set of edges of $G$ and the number of vertices in $G$, respectively. Given a graph $G$ and an induced subgraph $H$ of $G$, we use $G \setminus H$ to denote the induced subgraph of $G$ on the vertices $V(G) - V(H)$. A star graph is a graph isomorphic to $K_{1,n-1}$.
The chromatic discrepancy parameters discussed are formally defined below.

\begin{definition}
Let $\calh$ be the set of all induced subgraphs and $C$ be the set of all proper
colorings of $G$.
We define $\varphi_c(G)$, for $c \in C$,
as
$\underset{H\in\calh}{\max}(|c(H)|-\chi(H))$, where $c(H)$ is the set of colors used by $c$ on $H$.
The \textbf{chromatic discrepancy} of $G$, denoted by $\varphi(G)$, is defined as
$\underset{c\in C}{\min}\,\varphi_c(G)$. When $\calh$ is restricted to the set of all connected induced subgraphs of $G$, then the corresponding parameter is denoted by $\hatphi(G)$.
\end{definition}

Note that if $G$ is a disconnected graph then $\hatphi(G)$ is the maximum of $\hatphi(G_i)$ where $G_{i}s$ are the connected components of $G$.
It is clear from  the definition that
$\varphi(G) \ge \hatphi(G)$. It is also clear from the definition that both $\varphi$ and $\hatphi$ are monotonically non-decreasing, i.e.,
$\varphi(H)\leq \varphi(G)$ and $\hatphi(H)\leq \hatphi(G)$ for any induced subgraph $H$ of $G$.
Observe that if $G$ is a complete graph, then $\varphi(G) = \hatphi(G) = 0$ and if $G$ is an odd cycle of length more than $7$,
then\footnote{Let $G$ be an odd cycle $u_0u_1\ldots u_{n-1}u_0$ of length at least 9. Consider any proper coloring $c$ of $G$.
Without loss of generality, we may assume that $u_0, u_1$ and $u_2$ are assigned 3 distinct colors. Consider two vertices
$v, w\in S = \{u_4, u_5,\ldots, u_{n-2}\}$, such that $(v,w) \notin E(G)$ and $c(v) \neq c(w)$. These two vertices exist since $|S| \geq 4$.
Let $x\in \{u_0, u_1, u_2\}$ such that $c(x) \neq c(v)$ and $c(x) \neq c(w)$. Then $H = \{x, v, w\}$ is an
independent set in $G$ in which each vertex has a different colour. For this independent set $H$, we have $\varphi_c(H) = 2$.
}
$\varphi(G) = 2$ and $\hatphi(G) = 1$.
The smallest graph $G$ where $\varphi$ and $\hatphi$ differ is $K_2 \cup K_1$, i.e., the disjoint union of $K_2$ and $K_1$. For this graph, $\varphi(G)=1$ and $\hatphi(G)=0$.

The notion of \emph{local chromatic number} $\psi(G)$ of a graph $G$ was introduced by Erd{\"o}s et al. in \cite{erdos1986coloring} and is defined as the minimum over all proper colorings of $G$, the maximum number of colors present in the closed neighborhoods of vertices in $G$. It was shown in \cite{erdos1986coloring} that there are graphs with local chromatic number $3$ and chromatic number greater than $k$ for any positive integer $k$.
In Section~\ref{section:lb}, we obtain a lower bound on chromatic discrepancy of triangle-free graphs using local chromatic number. As the local chromatic number is bounded from below by the \emph{fractional chromatic number} (see \cite{korner2005local}), we obtain a lower bound on chromatic discrepancy using fractional chromatic number as well. Another related notion is \emph{perfect coloring} of a graph defined in \cite{sandeep2011perfectly}. A proper coloring
of a graph $G$ is known as perfect coloring if every connected induced subgraph $H$ of $G$ uses exactly $\omega(H)$ number of colors, where $\omega(H)$ is the clique number of $H$. We prove in Section~\ref{section:str} that the graphs which admit such a coloring are exactly the graphs for which $\hatphi(G)=0$. We remark that our parameter $\varphi$ differs from the \emph{hypergraph (combinatorial) discrepancy} \cite{matousek1999geometric} which measures the minimum over all two-colorings of a hypergraph, the maximum imbalance between the cardinality of the color classes of vertices in the hyperedges of the hypergraph.

Finding a coloring that achieves chromatic discrepancy has potential applications in  channel allocation problems in wireless networks \cite{ramachandran2006interference}.
The wireless network can be represented as a graph $G$ whose vertices correspond to the various network devices and the edges correspond to communication links between pairs of devices. The goal is to assign channels to communication links without causing any radio interference.
A conflict graph $G_C$ of $G$ has vertices corresponding to the edges in $G$ and an edge is present between two vertices in $G_C$ if assigning the same channel to the corresponding communication links can lead to a radio interference.
Allocating channels to the communication links is solved using a vertex coloring of $G_C$, where the set of colors correspond to the set of channels used in the network.
Each channel uses a separate frequency band from the available communication band. By reducing the total number of channels (colors), each channel can be assigned a larger bandwidth for faster communication.
In addition, reducing the number of different channels (colors) used in subgraphs corresponding to various geographic regions can potentially have two benefits: (a) communication links in these regions can use larger bandwidth and (b) more channels (colors) are available for channel allocation in other co-located wireless networks in these regions.

\section{Our Results}
We give upper bounds on $\varphi(G)$ in terms of chromatic number $\chi(G)$ and independence number $\alpha(G)$.
We show that for a graph on $n$ vertices, $\varphi(G)$ is at most $ \min\{n/3, ~\chi(G) (1 - 1/\alpha(G)), ~n - \chi(G)\}$. This upper bound is tight for complete graphs and graphs of the form $K_t \cup K_{2t}$.
We show that for any graph $G$,  $\varphi(G)$ is lower bounded by  $(\chi(G)-\omega(G))/2$, where $\omega(G)$ is the clique number. As a consequence, $\varphi(G) \ge \chi(G)/2 -1 $ for triangle free graphs.

For Mycielski graphs $M_k$ of order $k$, we show that $\varphi(M_k) = \hatphi(M_k) = \chi(M_k) - 2  = k - 2$. Hence, for Mycielski graphs, the above upper bound is tight up to additive constant and the above lower bound is tight up to multiplicative constant. We also show that for triangle free graphs, $\varphi(G) \ge \hatphi(G) \ge \psi(G)-2$ where $\psi(G)$ is the local chromatic number of $G$.  We obtain a lower bound on the chromatic discrepancy parameters for random graphs under the $G(n, p)$ model. We show that for $2\log (n)/n < p < 1/\log^2(n)$ and for a constant $C > 0$, a.a.s., $\varphi(G) \ge \hatphi(G) \ge \chi(G) (1 - C/\log(np))$.

We provide structural characterization for graphs with $\hatphi(G) = 0$ and graphs with $\varphi(G) = 0$. We show that graphs with $\hatphi(G) = 0$ are exactly the class of paw-free perfect graphs (See Figure \ref{fig:paw} for paw graph)  and graphs with $\varphi(G) = 0$ are exactly the class of perfect graphs which are complete multipartite.
In general, $\varphi(G)$ and $\hatphi(G)$ of perfect graphs need not be bounded. We show that there are perfect graphs with arbitrarily large $\varphi(G)$ and similarly there are perfect graphs with arbitrarily large $\hatphi(G)$.

We also obtain bound on the gap between $\varphi(G)$ and $\hatphi(G)$. We show that $\varphi(G) \le \hatphi(G) + \alpha(G) - 1$ for any connected graph $G$.
It is natural to ask whether there exists an optimal coloring of $G$ that achieves $\varphi(G)$ or $\hatphi(G)$.
We show a class of graphs for which $\varphi_c(G)$ and $\hatphi_c(G)$ obtained by any optimal coloring $c$ is arbitrarily large compared to $\varphi(G)$ and $\hatphi(G)$. From the computational aspect we show that computing $\varphi(G)$ and $\hatphi(G)$ are NP-hard.

\section{Upper Bounds}\label{section:ub}
In this section, we derive upper bounds for the chromatic discrepancy parameters $\varphi(G)$ and $\hatphi(G)$ in terms of other graph parameters.

In an optimal coloring $c$ of a graph $G$, $\varphi_c(G)$ is maximized if there is an independent set on $\chi(G)$ vertices of distinct colors. Similarly, $\hatphi_c(G)$ is maximized if there is a connected induced 2-chromatic subgraph containing vertices of all the colors. Thus we have the following proposition.
\begin{proposition}\label{prop:trivialubound}
For any graph $G$, $\varphi(G)\leq \chi(G)-1$ and for any graph $G$ with at least one edge, $\hatphi(G)\leq \chi(G)-2$\,.
\end{proposition}

The following proposition is widely used in the rest of the paper.
\begin{lemma}\label{lem:ind}
Let $G$ be any graph and let $I$ be an independent set in $G$.
Then:
\[ \varphi(G) \leq \varphi(G \setminus I) + 1\,. \]
In particular, if $H$ is an induced subgraph of $G$,
\[ \varphi(G) \leq \varphi(H) + \chi(G \setminus H)\,. \]
\begin{proof}
Consider any proper coloring $c$ of $G\setminus I$. Give a new color to vertices in $I$. Let the coloring obtained for $G$ be $c^\prime$.
It is easy to see that $\varphi_{c^\prime}(G)\leq \varphi_c(G\setminus I)+ 1$. By choosing $c$ to be the coloring of $G\setminus I$ for which $\varphi_c(G\setminus I)$ is minimum, we get
$\varphi_{c'} (G) \leq \varphi(G \setminus I) + 1$. Thus we obtain $\varphi(G) \leq \varphi(G \setminus I) + 1$.
By removing each color class of an optimal coloring of $G\setminus H$ and applying the upper bound obtained above, we obtain $\varphi(G) \leq \varphi(H) + \chi(G \setminus H)$.
\end{proof}
\end{lemma}

The following lemma comes handy in proving our results.

\begin{lemma}\label{lem:colorclass}
Let $c$ be a proper vertex coloring of a graph $G$. There exists an induced subgraph $H$ of $G$ such that $H$ contains exactly one
vertex from each color class and $\varphi_c(G) = |c(H)| - \chi(H)$\,.
\begin{proof}
Consider an induced subgraph $H^\prime$ of $G$ with minimum number of vertices such that $|c({H^\prime})|-\chi(H^\prime) = \varphi_c(G).$ Clearly $H^\prime$ cannot contain more than one vertex of the same color class. Now, obtain an induced subgraph $H$ from $H^\prime$ by including
exactly one vertex from each color class which is absent in $H^\prime$. This induced subgraph has the properties stated in the lemma.
\end{proof}
\end{lemma}

We now give upper bounds for $\varphi(G)$, and thereby also for $\hatphi(G)$, in terms of the number of vertices $n$, chromatic number $\chi(G)$ and the independence number $\alpha(G)$ of $G$.

\begin{theorem}\label{thm:ub}
\begin{enumerate}[{(}i)]\label{thmitem:main}
\item\label{thmitem:nchi} $\varphi(G) \leq n-\chi(G)$,
\item\label{thmitem:chialpha} $\varphi(G) \leq \chi(G) \left(1-\dfrac{1}{\alpha(G)}\right)$,
\item\label{thmitem:nby3} $\varphi(G) \leq n/3$.
\end{enumerate}
\begin{proof}
Consider an optimal coloring $c$ of $G$. By Lemma~\ref{lem:colorclass}, there exists an induced subgraph $H$ of $G$ with $\chi(G)$ vertices such that $\varphi_c(G)=|c(H)|-\chi(H)=\chi(G)-\chi(H)$. Observing that  $\varphi(G) \le \varphi_c(G)$ and that $\chi(G) \leq \chi(H) + \chi(G \setminus H)$, we have

\begin{equation}\label{eq:1}
\varphi(G) \leq \chi(G)-\chi(H) \leq \chi(G \setminus H)
\end{equation}

Now (\ref{thmitem:nchi}) follows from (\ref{eq:1}) by observing that $\chi(G \setminus H)$ is upper bound by the number of vertices in $G \setminus H$ which is $n - \chi(G)$.

Since $\chi(H) \geq \dfrac{|H|}{\alpha(H)}\geq \dfrac{|H|}{\alpha(G)} = \dfrac{\chi(G)}{\alpha(G)}$, (\ref{thmitem:chialpha}) also follows from (\ref{eq:1}).

We prove (\ref{thmitem:nby3}) by induction on the number of vertices, $n$.
The statement can be easily verified when $n \leq 2$.
Consider a graph $G$ on $n$ vertices, where $n \geq 3$.
If $G$ has an independent set $I$ of size 3, then by Proposition~\ref{lem:ind}, and using the induction assumption,
$\varphi(G) \leq \varphi(G \setminus I)+1 \leq \dfrac{n-3}{3}+1 = \dfrac{n}{3}$.

Otherwise, $\alpha(G) \leq 2$ and hence from  (\ref{thmitem:chialpha}), we obtain
\begin{equation}\label{eq:3}
\varphi(G) \leq \dfrac{\chi(G)}{2}
\end{equation}

From (\ref{thmitem:nchi}) and (\ref{eq:3}), it follows that $\varphi(G) \le \min\left\{\chi(G)/2, n - \chi(G)\right\} \le n/3$.
\end{proof}
\end{theorem}

If $\chi(G)$ is large, Theorem~\ref{thm:ub}(\ref{thmitem:nchi}) gives a good upper bound and if $\alpha(G)$ is small Theorem~\ref{thm:ub}(\ref{thmitem:chialpha}) gives a good upper bound. Note that both Theorem~\ref{thm:ub}(\ref{thmitem:nchi}) and (\ref{thmitem:chialpha}) are tight for complete graphs. Theorem~\ref{thm:ub}(\ref{thmitem:nby3}) is tight for the graphs of the form $K_t\cup K_{2t}$.

\section{Lower Bounds}\label{section:lb}
First, we obtain a lower bound for chromatic discrepancy $\varphi(G)$ of a graph $G$ in terms of chromatic number and clique number $\omega(G)$.
\begin{theorem}\label{thm:lb1}
Let $\calh$ be the set of all induced subgraphs of a graph $G$. Then $\varphi(G) \geq \underset{H\in\calh}{\max} ~~\dfrac{1}{2}\left(\chi(H)-\omega(H)\right)$.
\end{theorem}

\begin{proof}
Consider any proper coloring $c$ of $G$ and  consider a set of $\chi(G)$ vertices of distinct colors. Let $H$ be the graph induced on these vertices.
It is known from \cite{rabern2008note} that $\chi(H) \leq (|H|+\omega(H))/2$. Hence $\varphi_c(G) \ge |H| - \chi(H) \geq (|H|-\omega(H))/2$.
Since $|H|=\chi(G)$ and $\omega(H) \leq \omega(G)$ it follows that $\varphi_c(G)\geq (\chi(G)-\omega(G))/2$ for any coloring $c$ and thus $\varphi(G)\geq (\chi(G)-\omega(G))/2$.  The result now follows from the monotonicity of chromatic discrepancy.
\end{proof}

From the above lower bound it follows that for any triangle-free graph $G$, $\varphi(G)\geq \chi(G)/2 -1$.

We use the following lemma to derive our next lower bound.

\begin{lemma}[Haxell, \cite{haxell2001note}]\label{lem:Haxell}
Let $G$ be a graph with maximum degree $\Delta$.
Let $V_1,V_2,\ldots,V_r$ be pairwise disjoint subsets of $G$ such that  $|V_i| \geq 2\Delta$ for all $i$.
Then there exists an independent set $S=\{v_1,\ldots,v_r\}$ such that $v_i \in V_i$ for each $i$.
\end{lemma}

\begin{theorem}\label{thm:lb}
Let $G$ be a graph with maximum degree $\Delta$. Let $\alpha$ and $\chi$ denote the independence number and chromatic number of $G$ respectively.
If $\alpha \geq 2\Delta$ then $\varphi(G) \geq \dfrac{n-(2\Delta -1)(\chi-1)}{\alpha} - 1$.
\end{theorem}
\begin{proof}
Consider a coloring $c$ of $G$ which achieves $\varphi(G)$.
First we observe that the number of colors used by $c$ is at most $\chi + \varphi(G)$. Otherwise it follows from the definition of $\varphi(G)$ that   $\varphi_c(G) \ge |c(G)| - \chi > \varphi(G)$. This contradicts the assumption that $\varphi_c(G) = \varphi(G)$.

Let $V_1,\ldots, V_r$ be the color classes which have
size at least $2\Delta$.
Note that $r \geq 1$ because $\alpha \geq 2 \Delta$.
By applying Haxell's lemma, we obtain an independent set of $r$ vertices (one from each color class) with distinct colors. Since the chromatic number of independent set is $1$, it follows that
\begin{equation}\label{eq:lb_1}
\varphi(G) \geq r-1
\end{equation}

Since the size of any color class is at most $\alpha$, we obtain $|V_1| + \ldots + |V_r| \le r \alpha$. Consider the remaining color classes. Each of these color classes have at most $2\Delta - 1$ vertices. Recalling that the number of colors is at most $\chi + \varphi(G)$, it follows that the sum total of the number of vertices in these remaining color classes is at most $(2 \Delta - 1)(\chi + \varphi(G) - r)$. We can now bound $n$ as
 $n \le r\alpha + (2\Delta - 1)(\chi + \varphi(G) - r)$. Since $\alpha \geq 2\Delta$, this bound can be restated as the following lower bound on $r$:
\begin{equation}\label{eq:lb_2}
r ~\ge~ \frac{n - (2\Delta - 1) (\chi + \varphi(G))}{\alpha - 2\Delta + 1}
\end{equation}
The result now follows from (\ref{eq:lb_1}) in a straightforward manner by substituting for $r$  using (\ref{eq:lb_2}) and by using the assumption that $\alpha \ge 2\Delta$.

\end{proof}

The following corollary is a direct consequence of Theorem \ref{thm:lb} and the fact that $\chi \le \Delta + 1$.

\begin{corollary}
Let $G$ be a graph with maximum degree $\Delta$ and independence number $\alpha$.
If $\alpha \geq 2\Delta$ and $\Delta \leq \sqrt{n}/2$ then $\varphi(G) \geq \dfrac{n}{2\alpha}-1$.
\end{corollary}

In the following, we derive lower bounds for $\varphi(G)$ and $\hatphi(G)$ for Erd\H{o}s-R{\'e}nyi random graphs $G(n,p)$. In order to prove the desired lower bound, we require the following lemma.

\begin{lemma}\label{lem:lb}
Let $G$ be a connected graph having diameter $\mbox{diam}(G)$.
Let every induced subgraph $H$ of $G$ having  at most $\chi(G) \cdot \mbox{diam}(G)$ vertices satisfy the condition that
$\chi(H) \leq k$ for some fixed $k$. Then $\hatphi(G) \geq \chi(G) - k$.
\end{lemma}
\begin{proof}
Consider any proper coloring $c$ of $G$. Consider any $\chi(G)$ distinctly colored vertices and let $H$ be the induced subgraph on these vertices.
Since graph $H$ need not be connected,
we construct a connected induced subgraph $H'$ of $G$ from $H$ by including additional vertices. A vertex in $H$ can be connected to any other vertex in $H$ by including at most $\mbox{diam}(G) - 1$ additional vertices. It is straightforward to see that
$H$ can hence be extended to a connected induced subgraph $H'$ by including at most $(\chi(G) - 1)(\mbox{diam}(G)-1)$ additional vertices.
Since $H'$ satisfies that $|H'| \le \chi(G) \cdot \mbox{diam}(G)$, we have $\chi(H') \le k$. Since $H'$ includes $\chi(G)$ distinct colored vertices, it follows that $\hatphi_c(G) \ge \chi(G) - \chi(H') \ge \chi(G) - k$. Since $\hatphi_c(G) \ge \chi(G) - k$ for any coloring $c$, the result follows.
\end{proof}

The following theorem shows that for Erd\H{o}s-R{\'e}nyi random graphs $G(n,p)$,  $\varphi(G)$ and $\hatphi(G)$ are close to $\chi(G)$.
We prove the theorem using  Lemma \ref{lem:lb} and using the fact that for random
graphs $G(n, p)$, the values of chromatic number and diameter are concentrated and are known asymptotically (see \cite{chung2001diameter, luczak1991chromatic}). We use the following standard terminology: we say that an event occurs \emph{asymptotically almost surely} (and denote it by a.a.s.) when the probability of the occurrence of the event tends to 1 as $n \rightarrow \infty$.

\begin{theorem}\label{thm:rg}
Let $G$ be a random graph drawn under the $G(n, p)$ model with $\frac{2\log n}{n} < p < \frac{1}{\log^2 n}$.
Then for some constant $C>0$, a.a.s., $G$ satisfies
\[ \varphi(G) ~\geq~ \hatphi(G) ~\geq~ \chi(G)\left(1-\frac{C}{\log {np}}\right). \]
\end{theorem}

\begin{proof}
We use the following results from \cite{chung2001diameter, luczak1991chromatic, janson2011random, Friedgut03asharp}.

Consider a random graph $G$ drawn under the $G(n,p)$ model with $\frac{2\log n}{n} < p < \frac{1}{\log^2 n}$. Then
\begin{enumerate}[{(}i)]
\item A.a.s., $G$ is connected \cite{Friedgut03asharp}.
\item\label{seconditem} A.a.s.  $\frac{np}{2\log {np}}\leq\chi(G)\leq \frac{np}{\log {np}}(1+o(1))$ \cite{luczak1991chromatic}.
\item\label{thirditem} A.a.s. the diameter $\mbox{diam}(G)$ is at most $\frac{\log n}{\log np}(1+o(1))$ (Theorem 3 in \cite{chung2001diameter}).
\item\label{fourthitem} A.a.s. every subgraph $H$ of $G$ of size at most $n/\log^2 (np)$ vertices satisfies $\chi(H) \leq 1 + 2np / \log^2 (np)$ (Lemma 7.6 and 7.7 in \cite{janson2011random}).
\end{enumerate}

It is straightforward to verify that $\chi(G) \cdot \mbox{diam}(G) \le n/\log^2 (np)$ using the fact that $p < 1/\log^2 n$. We now apply Lemma \ref{lem:lb} to $G$ by fixing $k = 1 + 2np / \log^2 (np)$.
\begin{align*}
\hatphi(G)&\geq\chi(G)-k\\
&=\chi(G)-\left(1+\dfrac{2np}{\log^2 {np}}\right)\\
&=\chi(G)-\left(\dfrac{\chi(G)}{\log {np}}\right)\left(\dfrac{\log {np}}{\chi(G)}\right)\left(1+\dfrac{2np}{\log^2 {np}}\right)\;.
\end{align*}
Now we observe that $\dfrac{\log {np}}{\chi(G)}\left(1+\dfrac{2np}{\log^2 {np}}\right) = O(1)$.
\begin{align*}
\dfrac{\log np}{\chi(G)}\left(1+\dfrac{2np}{\log^2 {np}}\right)&\leq \dfrac{\log {np}}{\left(\dfrac{np}{2\log {np}}\right)}\left(1+\dfrac{2np}{\log^2 {np}}\right)\\
&=\dfrac{2}{np}\left(\log^2 {np} + 2np\right)\\
&=O(1)\;,
\end{align*}
where the lower bound on $\chi(G)$ (as stated in (\ref{seconditem}) above) is used for the inequality above.

Using the above calculations, we infer that $\hatphi(G) \geq \chi(G)\left(1-\frac{C}{\log {np}}\right)$.  Since $\varphi(G) \ge \hatphi(G)$ the result follows.
\end{proof}

Recall that the local chromatic number $\psi(G)$ of a graph $G$ is defined as the minimum over all proper colorings of $G$, the maximum number of colors present in the closed neighborhoods of vertices in $G$.

\begin{proposition}\label{prop:local}
For a triangle-free graph $G$, $\varphi(G) \geq \hatphi(G) \geq \psi(G)-2$.
\end{proposition}
\begin{proof}
By the definition of $\psi(G)$, in any proper coloring of $G$, there exists a vertex $v$ such that the closed neighborhood of $v$ uses at least $\psi(G)$ colors. Since $G$ is triangle-free, the closed neighborhood of any vertex forms a connected bipartite graph whose chromatic number is $2$. Hence the result.
\end{proof}
It is known from \cite{korner2005local} that the local chromatic number of a graph $G$ is bounded from below by fractional chromatic number $\chi^*(G)$ of $G$, which is a relaxed version of chromatic number. Thus, for a triangle-free graph, we get  $\varphi(G)\geq\hatphi(G)\geq \chi^*(G)-2$.
\section{Structural Characterization}
\label{section:str}
Here we characterize the graphs $G$ with $\varphi(G)=0$. We also characterize the graphs $G$ with $\hatphi(G)=0$.
\begin{proposition}\label{prop:prop2}
If $G$ is a paw graph (Figure~\ref{fig:paw})
or an odd cycle of length at most $7$ then  $\hatphi(G) = \varphi(G) = 1$. If $G$ is an odd cycle of length more than $7$ then $\varphi(G) = 2 $ and $\hatphi(G)=1$.
\end{proposition}

\begin{figure}[htp]
 \begin{center}
  \includegraphics[scale=0.4]{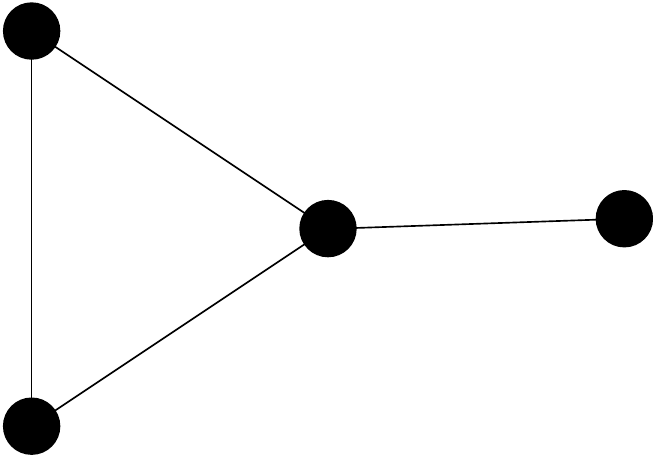}
  \caption{Paw graph}
\label{fig:paw}
 \end{center}
\end{figure}

\begin{theorem}\label{thm:char}
$\hatphi(G) = 0$ if and only if $G$ is paw-free and perfect. $\varphi(G) = 0$ if and only if $G$ is complete multipartite.
\end{theorem}
\begin{proof}
Let $G$ be paw-free and perfect. Then $G$ is perfectly colorable \cite{sandeep2011perfectly}, i.e., there exists a proper coloring for
$G$ such that every connected induced subgraph $H$ of $G$ uses $\omega(H)=\chi(H)$ colors. That is  $\hatphi(G) = 0$.
For the only if part consider a graph $G$ with  $\hatphi(G)=0$.
It is known from strong perfect graph theorem that a graph is perfect if and only if it does not contain any induced odd cycle of length at least $5$ or complement of an odd cycle of length at least $5$.
From Proposition \ref{prop:prop2}, it is clear that $G$ should not contain a paw or odd cycle of length at least $5$ as an induced subgraph.
It is easy to see that complement of an odd cycle of length at least $7$ contains paw as an induced subgraph. It follows that $G$ should be
perfect and paw-free.

Now consider the second part of the theorem.
If $G$ is complete multipartite then an optimal coloring achieves $\varphi(G) = 0$.
For the only if part consider a graph $G$ with  $\varphi(G)=0$.
Recalling that $\varphi(K_1\cup K_2) = 1$ it follows that $G$ does not have $K_1\cup K_2$ as an induced subgraph.
It is straightforward to verify that a graph is complete multipartite if and only if
it does not have $K_1\cup K_2$ as an induced subgraph. This implies that $G$ is complete multipartite.
\end{proof}

It is proved in \cite{olariu1988paw} that every component of a paw-free perfect graph is either bipartite or complete multipartite. Hence we have the following corollary.
\begin{corollary}
$\hatphi(G) = 0$ if and only if every component of $G$ is either bipartite or complete multipartite.
\end{corollary}

In the above characterization, we showed that graphs with $\hatphi(G) = 0$ are exactly the class of paw-free perfect graphs and graphs with $\varphi(G) = 0$ are exactly the complete multipartite graphs, which form a subclass of perfect
graphs.
In the following, we show that $\varphi(G)$ and $\hatphi(G)$ of perfect graphs need not be bounded. In particular, there are perfect graphs with arbitrarily large $\varphi(G)$ and similarly there are perfect graphs with arbitrarily large $\hatphi(G)$.

\begin{theorem}\label{thm:tight}
For any two integers $c$ and $p$, such that $c\geq 1$ and $0\leq p\leq c-1$, there exists a perfect graph $G$ with $\chi(G)=c$
and $\varphi(G)=p$. Similarly, for any two integers $c$ and $p$ such that $c\geq 2$ and $0\leq p\leq c-2$, there exists
a perfect graph with $\chi(G)=c$ and $\hatphi(G)=p$.
\end{theorem}
\begin{proof}
Let $\cup$ and $\Join$ denote the standard graph operations disjoint union and join respectively. We now prove the first part of the theorem statement.
Let $p \cdot K_p$ denote the graph which is the disjoint union of $p$ copies of $K_p$.
Consider a graph $G$ which is $K_c\cup p\cdot K_p$ (see Figure~\ref{fig:kcupkp}). That is, $G$ is the disjoint union of $K_c$ and $p$ copies of $K_p$.
Clearly $G$ is a perfect graph and $\chi(G) = c \ge p+1$.
By Proposition~\ref{lem:ind}, we have $\varphi(G)\leq \varphi(K_c) + \chi(p.K_p)$. Since $\varphi(K_c) = 0$, it follows that $\varphi(G) \le p$.
In any proper coloring $r$ of $G$,  we can obtain an independent set of $p+1$ vertices with distinct colors, by choosing one vertex from
each $K_p$ and one vertex from $K_c$. Hence $\varphi_r(G) \ge p$ for any coloring $r$. It follows that $\varphi(G) \ge p$. Recalling that $\varphi(G) \le p$, we have $\varphi(G) = p$.

We now prove the second part.
Following the earlier notation, let $(p+1)\cdot K_{p+1}$ denote the disjoint union of $p+1$ copies of $K_{p+1}$.
Consider a graph $G$ which is $K_{c-(p+1)} \Join (p+1)\cdot K_{p+1}$ (see Figure~\ref{fig:kcpp1kpp1}). That is, $G$ is the join of $K_{c - (p+1)}$ and $(p+1) \cdot K_{p+1}$.
It is straightforward to verify that $G$ is a perfect graph and $\chi(G) = c \ge p+2$.
It remains to show that $\hatphi(G) = p$.
In any proper coloring $r$ of $G$, we can obtain an induced star graph of $p+2$ vertices with distinct colors, by choosing one vertex from each $K_{p+1}$ and one vertex from $K_{c - (p+1)}$. Hence $\hatphi_r(G) \ge p$ for any coloring $r$. It follows that $\hatphi(G) \ge p$.
Consider the optimal coloring $r^*$ of $G$ which assigns the same $p+1$ colors to each
copy of $K_{p+1}$ and remaining $c - (p+1)$ colors to $K_{c - (p+1)}$.
Consider any connected induced subgraph $H$ of $G$. Let $H$ contain $x$ vertices from $K_{c - (p+1)}$. Let $H$ contain a maximum of $y$ vertices from any one of the $p+1$ copies of $K_{p+1}$.  Clearly $\chi(H) = x + y$.  Observing that $H$ can contain at most $x + p+1$ colors of $r^*$, it follows that $|r^*(H)| - \chi(H) \le x + p + 1 - (x + y) \le p + 1 - y \le p$ if $y \ge 1$. If $y= 0$ then $H$ contain only vertices from $K_{c - (p+1)}$ and hence $|r^*(H) - \chi(H)| = 0 \le p$. It follows that $\hatphi_{r^*}(G) \le p$. Recalling that $\hatphi(G) \ge p$, we obtain that $\hatphi(G) = p$.
\end{proof}


\begin{figure}[h]
\centering
\subfloat[Subfigure 1 list of figures text][$G = K_{c} \cup p\cdot K_{p}$]{
\resizebox{5.4cm}{!}{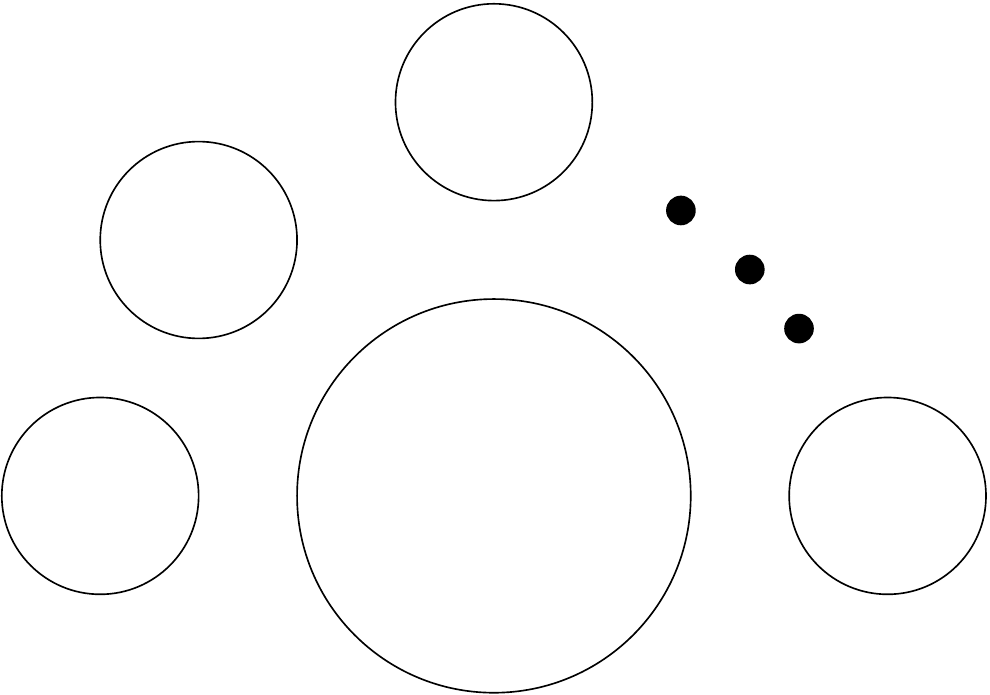}
\label{fig:kcupkp}}
\qquad
\subfloat[Subfigure 2 list of figures text][$G = K_{c-(p+1)} \Join (p+1)\cdot K_{p+1}$]{
\resizebox{5.4cm}{!}{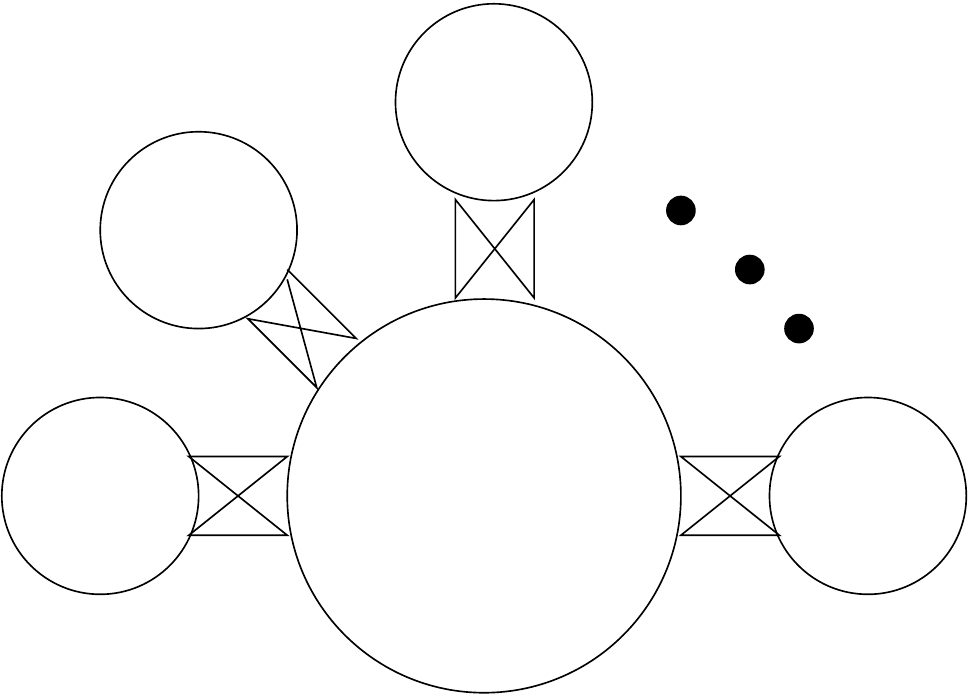}
\label{fig:kcpp1kpp1}}
\caption{Graph in (a) has $\chi(G) = c$ and $\varphi(G) = p$; Graph in (b) has $\chi(G) = c$ and $\hatphi(G) = p$.}
\end{figure}

\section{Relations between $\varphi(G)$ and $\hatphi(G)$}\label{section:on}

The following theorem gives a bound on the gap between $\hatphi(G)$ and $\varphi(G)$.
\begin{theorem}\label{thm:phialpha}
For a connected graph $G$,
$\varphi(G)\leq \hatphi(G)+\alpha(G)-1$.
\end{theorem}
\begin{proof}
Consider any proper coloring $c$ of $G$. Let $H$ be an induced subgraph
of $G$ with the largest number of vertices satisfying $\varphi_c(G)=|c(H)|-\chi(H)$.
We observe that $H$ is a dominating set of $G$. Otherwise, we can merely include in $H$ vertices that are not dominated by any vertex in $H$ without increasing its chromatic number. This contradicts the assumption that $H$ has largest number of vertices.

The graph $H$ need not be connected and  $H$ can have at most $\alpha(G)$ components.
We now show that if $H$ is not connected then $H$ can be extended to a connected induced subgraph $S$ by including additional vertices in a straightforward manner.
Let $G \setminus H$ denote the graph induced by vertices outside of $H$.
We recall that $G$ is connected and $H$ is a dominating set. Hence there are following two possibilities.
There exists a vertex $u$ in $G \setminus H$ and including it in $H$ connects two or more components of $H$ in the resulting induced subgraph $H'$.
If no such vertex exists in $G \setminus H$ then there exists an edge in $G \setminus H$ and including both endpoints of that edge
in $H$ connects exactly two components of $H$ in the resulting induced subgraph $H'$.
It is straightforward to verify that in either case $\chi(H') \le \chi(H) + 1$.
Proceeding in this fashion, we extend $H$ to a connected induced subgraph $S$ such that
$\chi(S) \le \chi(H) + \alpha(G) - 1$.
Recalling that $\varphi_c(G)=|c(H)|-\chi(H)$,
it follows that $\hatphi_c(G) \ge |c(S)| - \chi(S) \ge |c(H)| - \chi(H) - \alpha(G) + 1 = \varphi_c(G) - \alpha(G) + 1$.
Since $\varphi_c (G) \leq \hatphi_{c} (G)+\alpha(G)-1$ for any coloring $c$, by considering $c$ to be a
coloring of $G$ such that $\hatphi(G)=\hatphi_{c}(G)$, we get the desired result. $$\varphi(G) \leq \varphi_c (G) \leq \hatphi_c (G)+\alpha(G)-1 = \hatphi(G)+\alpha(G)-1\;.$$
\end{proof}

Theorem~\ref{thm:phialpha} implies that $\varphi(G)$ and $\hatphi(G)$ are close if $\alpha(G)$ is small. For a disconnected graph this may not hold true.
For instance,  let $G$ be of the form $K_t\cup K_t$ for $t\geq 4$. Then, $\varphi(G) = \lfloor t/2\rfloor$, $\hatphi(G) = 0$ and $\alpha(G)=2$.
In the following we derive another bound on $\varphi(G)$  in terms of $\hatphi(G)$ for a class of graphs.

\begin{theorem}\label{thm:phialpha-restricted}
There exists a class of graphs for which $\varphi(G) \le \hatphi(G) + \sqrt{\varphi(G)+1} - 1$.
\end{theorem}
\begin{proof}
Consider the class of graphs
formed by graphs $G_t$ , where $t$ is a perfect square, that have the following structure (see Figure~\ref{fig:gap}).
The graph $G_t$ contains one $K_t$ and $t$ copies of $K_{t-1}$. Each vertex in $K_t$ is connected to all vertices of a distinct $K_{t-1}$.
Clearly, $\chi(G_t) = t$. Consider any proper coloring $c$ of $G_t$. It is easy to see that there exists an independent set of $t$ vertices having distinct colors. Hence $\varphi_c(G_t) = t-1$. It is also straightforward to verify that there exists a connected induced subgraph  $G_{\sqrt{t}}$ in $G_t$ containing $t$ colors. Hence $\hatphi_c(G_t) \geq t-\sqrt{t}$.
It follows that for any coloring $c$, $\varphi_c(G_t)  \le  \hatphi_c(G_t) + \sqrt{t}-1 = \hatphi_c(G_t) + \sqrt{\varphi_c(G_t) + 1} - 1$.
Rearranging the terms, we get 
\begin{equation}\label{eqn:perfectsquare}
\varphi_c(G_t) - \sqrt{\varphi_c(G_t) + 1} \leq \hatphi_c(G_t)  - 1\;.
\end{equation}
By considering $c$ to be a
coloring of $G_t$ such that $\hatphi(G_t)=\hatphi_{c}(G_t)$, we get the desired result as follows.
\begin{eqnarray*}
\varphi(G_t) - \sqrt{\varphi(G_t) + 1} &\leq& \varphi_c(G_t) - \sqrt{\varphi_c(G_t) + 1} \\
& \leq & \hatphi_c(G_t)  - 1\\
& = & \hatphi(G_t)  - 1.
\end{eqnarray*}
The inequality in the first line follows since $f(x) = x - \sqrt{x + 1}$ is an increasing function when $x \geq 0$, the inequality in the second line follows
because of equation (\ref{eqn:perfectsquare}) and the equality in the last line follows by the choice of the coloring $c$.
\end{proof}
\begin{figure}[htp]
\begin{center}
\resizebox{5cm}{!}{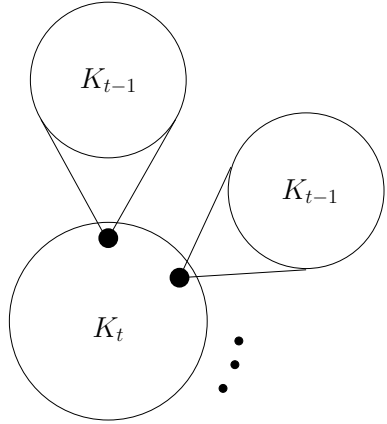}
\caption{Graph $G_t$ where $\varphi(G_{t}) \le \hatphi(G_t) +  \sqrt{t}-1$}
\label{fig:gap}
\end{center}
\end{figure}
\section{Additional Results}\label{section:other}

It is natural to ask whether there exists an optimal coloring of $G$ that achieves $\varphi(G)$ or $\hatphi(G)$.
In the following, we show a class of graphs for which $\varphi_c(G)$ and $\hatphi_c(G)$ obtained by any optimal coloring $c$ is arbitrarily large compared to $\varphi(G)$ and $\hatphi(G)$.
\begin{theorem} \label{thm:opt}
There exists a class of graphs for which the following holds. For any integer $r \ge 4$, there exist a graph $G$ in the class with $\varphi(G) = \hatphi(G) = 1$ whereas for any optimal coloring $c$ of $G$, $\varphi_c(G) = r-1$ and $\hatphi_c(G) = r-2$.
\end{theorem}
\begin{proof}
Consider a class of graphs that contains $G_r$ for each $r \ge 4$ and having the following structure (See Figure \ref{fig:globfig} for  $G_4$).
Graph $G_r$ contains $2r$ vertices of which $r$ vertices induce a $K_r$ and remaining $r$ vertices induce an independent set.
Each vertex in the independent set is connected to all except one distinct vertex in $K_r$. Clearly $\chi(G_r) = r$. Consider any optimal coloring $c$ of $G_r$.
Clearly all $r$ colors appear in the independent set. Similarly, all $r$ colors appear in the induced subgraph formed by a vertex in $K_r$ and all its neighbors in the independent set.
Using Proposition \ref{prop:trivialubound}, it follows that $\varphi_c(G_r) = r-1$ and $\hatphi_c(G_r) = r -2$. Consider a coloring $c'$ of $r+1$ colors that uses $r$ colors to color $K_r$ and uses extra one color for the independent set. It is straightforward to verify that for this coloring, $\varphi_{c'}(G_r) = \hatphi_{c'}(G_r) = 1$. Observing that $G_r$ contains an induced paw graph, from Proposition \ref{prop:prop2} it follows that $\varphi(G_r) \ge 1$ and  $\hatphi(G_r) \ge 1$. Since $\varphi(G_r) \le \varphi_{c'}(G_r)$ and $\hatphi(G_r) \le \hatphi_{c'}(G_r)$, we have $\varphi(G) = \hatphi(G) = 1$. Hence the result.
\end{proof}
\begin{figure}[h]
\centering
\subfloat[Subfigure 1 list of figures text][Unique optimal coloring $c$. $\varphi_c(G) = 3$ and $\hatphi_c(G) = 2$. Star subgraph is darkened.]{
\includegraphics[width=0.3\textwidth]{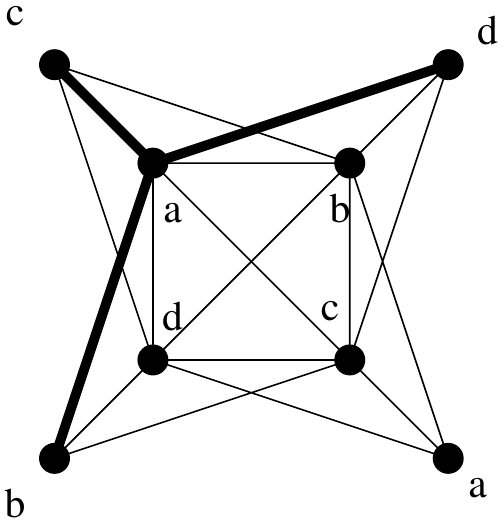}
\label{fig:subfig1}}
\qquad
\subfloat[Subfigure 2 list of figures text][Non optimal coloring $c^\prime$. $\varphi_{c^\prime}(G) = \hatphi_{c^\prime}(G) = 1$. Induced paw is darkened.]{
\includegraphics[width=0.3\textwidth]{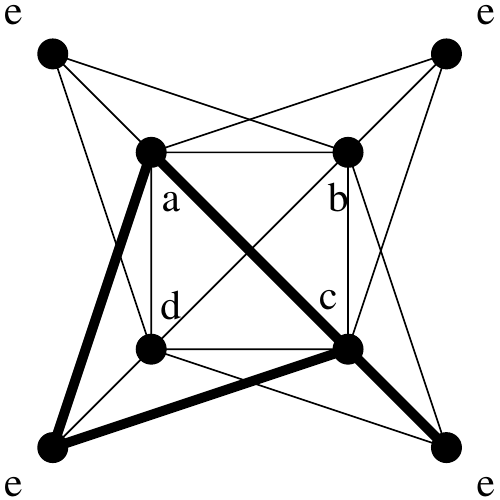}
\label{fig:subfig2}}
\caption{Graph $G_4$. No optimal coloring achieves $\varphi$ and $\hatphi$.}
\label{fig:globfig}
\end{figure}

A \emph{complete coloring} of a graph $G$ is a proper coloring such that every pair of distinct colors is used to color at least one adjacent pair of vertices. Clearly an optimal coloring is also a complete coloring. In Theorem \ref{thm:opt} we showed graph classes for which no optimal coloring achieves $\varphi(G)$ or $\hatphi(G)$. Following proposition states that for any graph $G$,  $\varphi(G)$ and $\hatphi(G)$ are achieved by complete colorings.
\begin{proposition}\label{prop:complete}
For any graph $G$, there exists a complete coloring $c$ with $\varphi_c(G) = \varphi(G)$ and there exists a complete coloring $c^{\prime}$ with $\hatphi_{c^{\prime}}(G) = \hatphi(G)$.
\end{proposition}
\begin{proof}
Assume by contradiction that there exists no complete coloring that achieves $\varphi(G)$. Consider a non-complete coloring $c$ which achieves $\varphi(G)$. By merging color classes that are mutually independent into a single color, we can arrive at a complete coloring $c'$ of $G$.
Clearly, for any induced subgraph $H$, we have $|c(H)| - \chi(H) \ge |c'(H)| - \chi(H)$ implying that $\varphi_c(G) \ge \varphi_{c'}(G)$, which is a contradiction. We can argue similarly for $\hatphi(G)$.
\end{proof}

We recall that the \emph{achromatic number} of a graph $G$ is the maximum number of colors that can be used in any complete coloring of $G$.
The following proposition is a direct consequence of Proposition \ref{prop:complete}
\begin{proposition}\label{prop:achro}
For a graph $G$, let $c$ (or $c'$) be a coloring with minimum number of colors that achieves $\varphi(G)$ ($\hatphi(G)$). Then $|c(G)|$ ($|c'(G)|$) is at most the achromatic number of $G$.
\end{proposition}


\subsection{Chromatic Discrepancy of Mycielski Graphs}

For an undirected graph $G$ on $n$ vertices and $m$ edges, the \emph{Mycielskian} $M(G)$ of $G$ \cite{mycielski1955, simonyi2006local} is defined as a graph on $2n + 1$ vertices and $3m + n$ edges and having the following structure. Let the vertex set $V(G)$ of $G$ be $\{v_1, \ldots, v_n\}$. The vertex set of $M(G)$ comprises three parts, namely,   $V = V(G)$,  $U = \{u_1, \ldots, u_n\}$ and  one extra vertex $w$.
The edge set of $M(G)$ contains (i) all edges of $G$, (ii) edge $(w, u_i)$ for each $u_i$ and (iii) two edges $(v_i, u_j), (v_j, u_i)$ for each edge $(v_i, v_j)$ in $G$.
It is known that $\chi(M(G)) = \chi(G)+1$ and if $G$ is triangle free then $M(G)$ is also triangle free. The \emph{Mycielski sequence} of graphs $M_2, M_3, \ldots $ is defined as: $M_2=K_2$, and  Mycielski graph of order $k$ for $k \ge 3$ is given by $M_k = M(M_{k-1})$. It is known that for any $k \ge 2$,  $M_k$ is triangle free and has chromatic number $k$.

\begin{theorem}\label{thm:myciel}
For Mycielski graph $M_k$ of order $k$, $\varphi(M_k) = \hatphi(M_k) = k-2$.
\end{theorem}
\begin{proof}
It is proved in \cite{simonyi2006local} that the local chromatic number for Mycielskian $M(G)$ of any graph $G$ is given by  $\psi(M(G)) = \psi(G)+1$. Since $\psi(M_2)=2$, it follows that $\psi(M_k) = k$ for all $k \ge 2$. Recalling that $M_k$ is triangle free, using Proposition~\ref{prop:local}, we obtain that $\varphi(M_k)\geq \hatphi(M_k) \geq k-2$. Since $\chi(M_k) = k$, by Proposition~\ref{prop:trivialubound} we have $\hatphi(M_k) = k-2$.

We now show that $\varphi(M_k) = k-2$.
Define an optimal coloring $C_k$ of $M_k$ recursively as follows.
Apply coloring $C_{k-1}$ to the induced subgraph $M_{k-1}$ in $M_k$.
We recall that the vertex set of $M_k$ comprises $V = \{v_1, \ldots, v_r\}$ corresponding to vertices of $M_{k-1}$, $U = \{u_1, \ldots, u_r\}$ and $w$.
Color vertex  $u_j$ with the color of $v_j$. Assign a new color to $w$.
It is straightforward to prove by induction for this coloring, there exists no independent set of $k$ vertices with distinct colors in $M_k$.
Since the number of colors used by $C_k$ is $k$, it follows that for any induced subgraph $H$ of $M_k$, $\varphi_{C_k}(H) \le k-2$.
Recalling that $\varphi(M_k) \ge k -2$, we obtain that $\varphi(M_k) = k-2$.
\end{proof}
\section{Complexity Aspects}\label{section:complex}
In this section we show that computing $\varphi(G)$ and $\hatphi(G)$ is NP-hard.  We show this using reduction from the problem of computing chromatic
number of a graph. We use the following lemma to prove the result.

\begin{lemma}\label{claim:nphard}
Let $G$ be a connected graph on $n\geq 2$ vertices. Consider a graph $G'$ which has one copy of $G$ and one copy of $K_{2n}$, and
each vertex in $G$ is adjacent to two distinct vertices in $K_{2n}$ (see Figure \ref{fig:nphard}). Consider any coloring of $G'$.
Let $G$ contain $t$ colors. Then we can obtain a connected induced subgraph $H$ of $G'$ such that $\chi(H) = t$ and $H$ contains $2t$ colors.
\end{lemma}
\begin{proof}
Since $G$ is connected, we have $2 \le \chi(G) \le t \le n$.
Let $C_G = \{1, 2, \ldots, t\}$ denote the set of colors in $G$. Clearly, there are at least $2n -t$ vertices in the $K_{2n}$ whose colors are different from those in $C_G$. Let $W$ denote the set of these vertices in the $K_{2n}$. Let $U$ denote the remaining set of (at most $t$) vertices in the $K_{2n}$.
\begin{figure}[htp]
\begin{center}
\resizebox{5cm}{!}{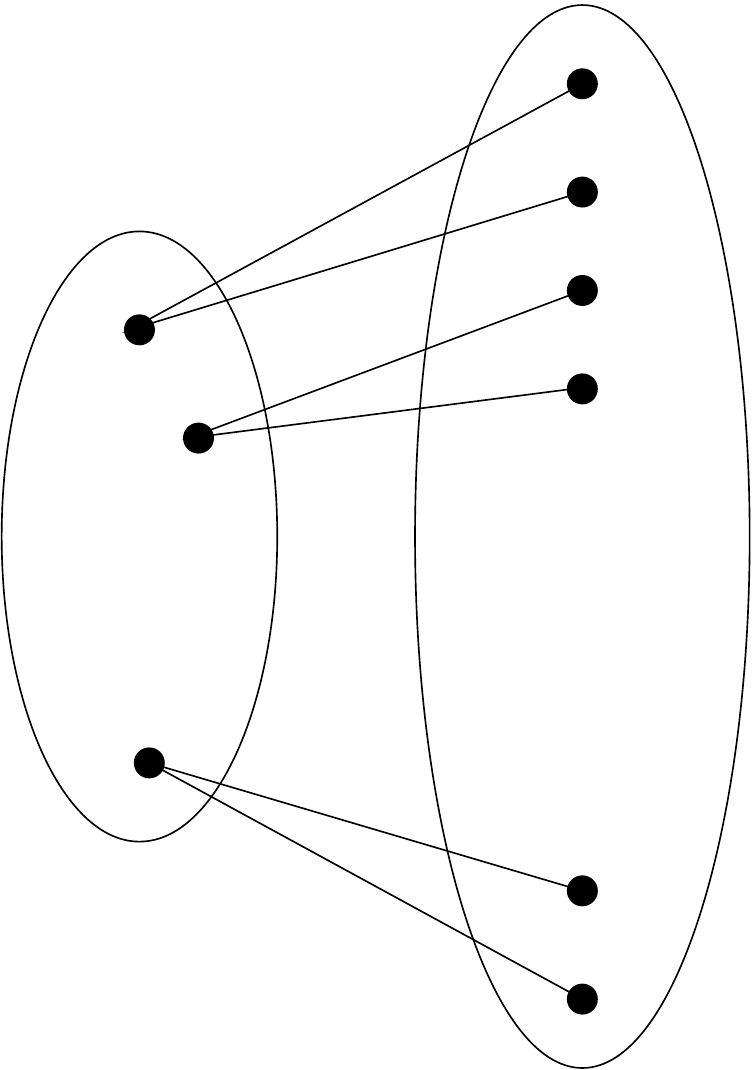}
\caption{Graph $G'$.}
\label{fig:nphard}
\end{center}
\end{figure}
We note that each vertex $v$ in the $K_{2n}$ has exactly one neighbor in $G$.
Hence, for each vertex $v$ in the $K_{2n}$, we define $\pi(v) \in C_G$ as the color of the neighbor of $v$ in $G$.
We now show that it is not possible to have same $\pi(w)$ value for every vertex $w \in W$.
By contradiction, we assume otherwise.
That is, all neighbors of $W$ in $G$ are colored with a single color from $\{1, \ldots, t\}$.
Hence the vertices in $G$ that are colored by the remaining $t-1$ colors are such that their neighbors in $K_{2n}$ is contained in $U$. Recalling that each vertex in $G$ has two neighbors in the $K_{2n}$ and that $|U| \le t$, it follows that $2(t -1) \le |U| \le t$, which implies that $t \le 2$. Recalling that $t \ge 2$, the only remaining possibility is that $t = |U| = 2$. This case can be restated as $G$ contains only two colors $\{1, 2\}$ and some vertex $v$ in $G$ has two neighbors in $U$ that are colored with distinct colors from $\{1, 2\}$, which is a contradiction.
Hence the neighbors of $W$ in $G$ contain two or more colors.

Recalling that $|W| \ge 2n -t$ and $t \le n$, we obtain that $|W| \ge t$.
It follows that
we can choose a set $S$ of $t$ vertices from $W$ such that their neighbors in $G$ contain two or more colors.
Clearly, vertices in $S$ have $t$ distinct colors and all of them are different from the $t$ colors $\{1, \ldots, t\}$ in $G$.
Let $H$ be the subgraph of $G^{\prime}$ induced by vertices of $G$ and $S$ together.
Clearly, $\chi(H) \geq t$ and $H$ contains $2t$ colors. It is also clear that $H$ is connected. We now conclude the proof by showing that $\chi(H) = t$ by exhibiting a $t$ coloring of $H$. Retaining the coloring of vertices in $G$, we color the vertices in $S$
one by one with any of the remaining valid colors from $\{1, \ldots, t\}$. We observe that to color the next vertex $w$ in $S$, we can use any color except $\pi(w)$ from the remaining colors. Hence $t - 1$ vertices in $S$ can be easily colored in this manner. Let $u$ be the last vertex  in $S$ to be colored and let $j$ be the single remaining color. If $j \not= \pi(u)$ then color $u$ with $j$. If $j = \pi(u)$ then choose another vertex $v$ in $S$ such that $\pi(u) \not= \pi(v)$. We recall that such a $v$ exists in $S$. Now the assign color of $v$ to $u$ and color $v$ with $j$. It is straightforward to verify that this is a valid $t$ coloring of $H$.
\end{proof}

\begin{theorem}\label{thm:nphard}
Given a graph $G$, finding $\varphi(G)$ and $\hatphi(G)$ are NP-hard.
\end{theorem}

\begin{proof}
We show a reduction from the problem of computing $\chi(G)$.
Without loss of generality, let $G$ be a connected graph on $n\geq 2$ vertices. Consider a graph $G'$ which has a copy of $G$ and a copy of $K_{2n}$, and
each vertex in $G$ is adjacent to two distinct vertices in $K_{2n}$ (see Figure \ref{fig:nphard}).
In the following we show that
$\varphi(G^{\prime}) = \hatphi(G^{\prime}) = \chi(G)$, which implies the result.
From Proposition~\ref{lem:ind}, it follows that $\varphi(G^{\prime}) \leq \varphi(K_{2n}) + \chi(G) = \chi(G)$.
Consider any coloring $c$ of $G^{\prime}$. Let $G$ contain $t$ colors.
Since $G$ is connected, we have $2 \le \chi(G) \le t \le n$.
Using Lemma \ref{claim:nphard} we obtain a connected induced subgraph $H$ of $G'$
such that $\chi(H) = t$ and $H$ contain $2t$ colors.
It follows that $\hatphi_c(G') \ge |c(H)| - \chi(H) = t$.
Since $\hatphi_c(G') \ge t$ for any coloring $c$, we obtain that $\hatphi(G') \ge t$. Recalling that $\varphi(G') \le \chi(G)$, $\chi(G) \le t$ and $\hatphi(G') \le \varphi(G')$,  it follows that $\varphi(G') = \hatphi(G') = \chi(G)$.
\end{proof}

It is known from \cite{zuckerman2006linear} that for all $\epsilon > 0$, approximating chromatic number within a factor $|V|^{1-\epsilon}$ is NP-hard.
The two chromatic discrepancy parameters inherit this inapproximability
result due to the same reduction as that is used in the proof of Theorem~\ref{thm:nphard}.
\section{Open Problems}\label{open}
We list some of the related open problems here.

\begin{enumerate}
\item Can we obtain, for general graphs $G$, a lower bound for $\hatphi(G)$ as a function of $\varphi(G)$ alone? Specifically, is it true that $\hatphi(G)\geq \varphi(G)+1-\sqrt{\varphi(G)+1}$?
\item Consider the following decision problems pertaining to $\varphi(G)$ and $\hatphi(G)$:
$$ L_1  = \{ \langle G, k \rangle | \; G \mbox{ is a graph such that } \varphi(G) \leq k \}. $$
$$ L_2 = \{ \langle G, k \rangle | \; G \mbox{ is a graph such that } \hatphi(G) \leq k \}. $$
We saw that $L_1$ and $L_2$ are NP-hard. Are $L_1$ and $L_2$ in NP?
\item Are there well-known graph classes for which $\varphi(G)$ and $\hatphi(G)$ are polynomial time computable?
\item Is it true that $\varphi(G)\geq \chi(G)-\omega(G)$? We note that for Mycielski sequence of graphs, where $\omega(G) = 2$, this is tight.
\end{enumerate}

\section*{Acknowledgement}
We gratefully acknowledge Manu Basavaraju, L. Sunil Chandran,  Mathew C. Francis and Bheemarjuna Reddy Tamma for fruitful discussions.




\bibliographystyle{elsarticle-num}
\bibliography{cd}







\end{document}

%% file: figs/kcupkp.pdf_t
\begin{picture}(0,0)%
\includegraphics{kcupkp.pdf}%
\end{picture}%
\setlength{\unitlength}{4144sp}%
\begingroup\makeatletter\ifx\SetFigFont\undefined%
\gdef\SetFigFont#1#2#3#4#5{%
  \reset@font\fontsize{#1}{#2pt}%
  \fontfamily{#3}\fontseries{#4}\fontshape{#5}%
  \selectfont}%
\fi\endgroup%
\begin{picture}(4516,3164)(3143,-3218)
\put(3556,-2401){\makebox(0,0)[b]{\smash{{\SetFigFont{14}{16.8}{\rmdefault}{\mddefault}{\updefault}{\color[rgb]{0,0,0}$K_p$}%
}}}}
\put(4051,-1231){\makebox(0,0)[b]{\smash{{\SetFigFont{14}{16.8}{\rmdefault}{\mddefault}{\updefault}{\color[rgb]{0,0,0}$K_p$}%
}}}}
\put(5401,-556){\makebox(0,0)[b]{\smash{{\SetFigFont{14}{16.8}{\rmdefault}{\mddefault}{\updefault}{\color[rgb]{0,0,0}$K_p$}%
}}}}
\put(7201,-2446){\makebox(0,0)[b]{\smash{{\SetFigFont{14}{16.8}{\rmdefault}{\mddefault}{\updefault}{\color[rgb]{0,0,0}$K_p$}%
}}}}
\put(5401,-2311){\makebox(0,0)[b]{\smash{{\SetFigFont{14}{16.8}{\rmdefault}{\mddefault}{\updefault}{\color[rgb]{0,0,0}$K_c$}%
}}}}
\end{picture}%

%% file: figs/kcpp1kpp1.pdf_t
\begin{picture}(0,0)%
\includegraphics{kcpp1kpp1.pdf}%
\end{picture}%
\setlength{\unitlength}{4144sp}%
\begingroup\makeatletter\ifx\SetFigFont\undefined%
\gdef\SetFigFont#1#2#3#4#5{%
  \reset@font\fontsize{#1}{#2pt}%
  \fontfamily{#3}\fontseries{#4}\fontshape{#5}%
  \selectfont}%
\fi\endgroup%
\begin{picture}(4426,3164)(3143,-3218)
\put(3556,-2401){\makebox(0,0)[b]{\smash{{\SetFigFont{12}{14.4}{\rmdefault}{\mddefault}{\updefault}{\color[rgb]{0,0,0}$K_{p+1}$}%
}}}}
\put(4051,-1231){\makebox(0,0)[b]{\smash{{\SetFigFont{14}{16.8}{\rmdefault}{\mddefault}{\updefault}{\color[rgb]{0,0,0}$K_{p+1}$}%
}}}}
\put(5401,-601){\makebox(0,0)[b]{\smash{{\SetFigFont{14}{16.8}{\rmdefault}{\mddefault}{\updefault}{\color[rgb]{0,0,0}$K_{p+1}$}%
}}}}
\put(7156,-2401){\makebox(0,0)[b]{\smash{{\SetFigFont{14}{16.8}{\rmdefault}{\mddefault}{\updefault}{\color[rgb]{0,0,0}$K_{p+1}$}%
}}}}
\put(5401,-2311){\makebox(0,0)[b]{\smash{{\SetFigFont{14}{16.8}{\rmdefault}{\mddefault}{\updefault}{\color[rgb]{0,0,0}$K_{c-(p+1)}$}%
}}}}
\end{picture}%

%% file: figs/gap.pdf_t
\begin{picture}(0,0)%
\includegraphics{gap.pdf}%
\end{picture}%
\setlength{\unitlength}{4144sp}%
\begingroup\makeatletter\ifx\SetFigFont\undefined%
\gdef\SetFigFont#1#2#3#4#5{%
  \reset@font\fontsize{#1}{#2pt}%
  \fontfamily{#3}\fontseries{#4}\fontshape{#5}%
  \selectfont}%
\fi\endgroup%
\begin{picture}(4277,4771)(4718,-4344)
\put(5851,-3391){\makebox(0,0)[b]{\smash{{\SetFigFont{20}{24.0}{\rmdefault}{\mddefault}{\updefault}{\color[rgb]{0,0,0}$K_t$}%
}}}}
\put(5851,-556){\makebox(0,0)[b]{\smash{{\SetFigFont{20}{24.0}{\rmdefault}{\mddefault}{\updefault}{\color[rgb]{0,0,0}$K_{t-1}$}%
}}}}
\put(8146,-1816){\makebox(0,0)[b]{\smash{{\SetFigFont{20}{24.0}{\rmdefault}{\mddefault}{\updefault}{\color[rgb]{0,0,0}$K_{t-1}$}%
}}}}
\end{picture}%

%% file: figs/nphard.pdf_t
\begin{picture}(0,0)%
\includegraphics{nphard.pdf}%
\end{picture}%
\setlength{\unitlength}{4144sp}%
\begingroup\makeatletter\ifx\SetFigFont\undefined%
\gdef\SetFigFont#1#2#3#4#5{%
  \reset@font\fontsize{#1}{#2pt}%
  \fontfamily{#3}\fontseries{#4}\fontshape{#5}%
  \selectfont}%
\fi\endgroup%
\begin{picture}(3436,4874)(2963,-4748)
\put(5581,-2986){\makebox(0,0)[b]{\smash{{\SetFigFont{12}{14.4}{\rmdefault}{\mddefault}{\updefault}{\color[rgb]{0,0,0}$\vdots$}%
}}}}
\put(3601,-2851){\makebox(0,0)[b]{\smash{{\SetFigFont{12}{14.4}{\rmdefault}{\mddefault}{\updefault}{\color[rgb]{0,0,0}$\vdots$}%
}}}}
\put(5581,-2311){\makebox(0,0)[b]{\smash{{\SetFigFont{12}{14.4}{\rmdefault}{\mddefault}{\updefault}{\color[rgb]{0,0,0}$K_{2n}$}%
}}}}
\put(3601,-2311){\makebox(0,0)[b]{\smash{{\SetFigFont{12}{14.4}{\rmdefault}{\mddefault}{\updefault}{\color[rgb]{0,0,0}$G$}%
}}}}
\end{picture}%

%% file: cd.bbl
\begin{thebibliography}{10}
\expandafter\ifx\csname url\endcsname\relax
  \def\url#1{\texttt{#1}}\fi
\expandafter\ifx\csname urlprefix\endcsname\relax\def\urlprefix{URL }\fi
\expandafter\ifx\csname href\endcsname\relax
  \def\href#1#2{#2} \def\path#1{#1}\fi

\bibitem{douglas1996west}
D.~B. West, Introduction to graph theory, 2nd Edition, Prentice Hall Inc.,
  Upper Saddle River, NJ.

\bibitem{erdos1986coloring}
P.~Erd\H{o}s, Z.~F{\"u}redi, A.~Hajnal, P.~Komj{\'a}th, V.~R{\"o}dl,
  {\'A}.~Seress, Coloring graphs with locally few colors, Discrete Mathematics
  59~(1) (1986) 21--34.

\bibitem{korner2005local}
J.~K{\"o}rner, C.~Pilotto, G.~Simonyi, Local chromatic number and {S}perner
  capacity, Journal of Combinatorial Theory, Series B 95~(1) (2005) 101--117.

\bibitem{sandeep2011perfectly}
R.~Sandeep, Perfectly colorable graphs, Information Processing Letters 111~(19)
  (2011) 960--961.

\bibitem{matousek1999geometric}
J.~Matousek, Geometric discrepancy. {A}n illustrated guide. {A}lgorithms and
  combinatorics, 18 (1999).

\bibitem{ramachandran2006interference}
K.~N. Ramachandran, E.~M. Belding-Royer, K.~C. Almeroth, M.~M. Buddhikot,
  Interference-aware channel assignment in multi-radio wireless mesh networks.,
  in: INFOCOM, Vol.~6, 2006, pp. 1--12.

\bibitem{rabern2008note}
L.~Rabern, A note on {R}eed's conjecture, SIAM Journal on Discrete Mathematics
  22~(2) (2008) 820--827.

\bibitem{haxell2001note}
P.~Haxell, A note on vertex list colouring, Combinatorics, Probability and
  Computing 10~(04) (2001) 345--347.

\bibitem{chung2001diameter}
F.~Chung, L.~Lu, The diameter of sparse random graphs, Advances in Applied
  Mathematics 26~(4) (2001) 257--279.

\bibitem{luczak1991chromatic}
T.~{\L}uczak, The chromatic number of random graphs, Combinatorica 11~(1)
  (1991) 45--54.

\bibitem{janson2011random}
S.~Janson, T.~{\L}uczak, A.~Rucinski, Random Graphs, Wiley Series in Discrete
  Mathematics and Optimization, Wiley, 2011.

\bibitem{Friedgut03asharp}
E.~Friedgut, V.~R{\"o}dl, A.~Rucinski, P.~Tetali, A sharp threshold for random
  graphs with a monochromatic triangle in every edge coloring, in: Memoirs of
  the AMS, 2006.

\bibitem{olariu1988paw}
S.~Olariu, Paw-free graphs, Information Processing Letters 28~(1) (1988)
  53--54.

\bibitem{mycielski1955}
J.~Mycielski, Sur le coloriage des graphes, Colloq. Math 3 (1955) 161--162.

\bibitem{simonyi2006local}
G.~Simonyi, G.~Tardos, Local chromatic number, {K}y {F}an's theorem, and
  circular colorings, Combinatorica 26~(5) (2006) 587--626.

\bibitem{zuckerman2006linear}
D.~Zuckerman, Linear degree extractors and the inapproximability of max clique
  and chromatic number, in: Proceedings of the thirty-eighth annual ACM
  symposium on Theory of computing, ACM, 2006, pp. 681--690.

\end{thebibliography}
